\documentclass[11pt]{amsart} \usepackage{latexsym, amssymb}

\newtheorem{thm}{Theorem}[section]
\newtheorem{lemma}[thm]{Lemma}
\newtheorem{prop}[thm]{Proposition}
\newtheorem{cor}[thm]{Corollary}
\newtheorem{question}[thm]{Question}
\newtheorem{fact}[thm]{Fact}

\theoremstyle{definition}
\newtheorem{df}[thm]{Definition}
\newtheorem{rmk}[thm]{Remark}
\newtheorem{rmks}[thm]{Remarks}
\newtheorem{ex}[thm]{Example}
\newtheorem{exs}[thm]{Examples}

\newcommand{\Z}{\mathbb{Z}}

\newcommand{\curly}[1]{\mathcal{#1}}
\newcommand{\A}{\curly{A}}

\newcommand{\cO}{\curly{O}}

\renewcommand{\o}{\Omega}

\def \r { {\mathbb R} }
\def \<{\langle}
\def \>{\rangle}
\def \n {\mathbb N}
\def \z {{\mathbb Z}}
\def \*Z {{{^*}\Z}}
\def \g{\mathfrak{g}}
\def \((  {(\!(}
\def \)) {)\!)}

\def \dh {\text d_H }

\def \st {\operatorname {st}}

\renewcommand{\dh}{\hat{d}}
\newcommand{\gh}{\hat{G}}
\newcommand{\hg}{\hat{\mathfrak{g}}}
\numberwithin{equation}{section}
\def \A{\mathcal{A}}

\def \u{\mathcal{U}}
\def \i{\operatorname{in}}
\def \sta{^\circ}
\def \tvs{\operatorname{tvs}}
\def \Diff{\operatorname{Diff}}
\def \VP{\operatorname{VP}}

\allowdisplaybreaks[2]

\begin{document}

\title{Nonstandard hulls of locally uniform groups}
\author{Isaac Goldbring}
\thanks{This work was partially supported by NSF grant DMS-1007144.}
\address{University of California, Los Angeles, Department of Mathematics\\ 520 Portola Plaza, Box 951555\\ Los Angeles, CA 90095-1555, USA}
\email{isaac@math.ucla.edu}
\urladdr{http://www.math.ucla.edu/~isaac}

\maketitle

\begin{abstract}
We present a nonstandard hull construction for locally uniform groups in a spirit similar to Luxembourg's construction of the nonstandard hull of a uniform space.  Our nonstandard hull is a local group rather than a global group.  We investigate how this construction varies as one changes the family of pseudometrics used to construct the hull.  We use the nonstandard hull construction to give a nonstandard characterization of Enflo's notion of groups that are uniformly free from small subgroups.  We prove that our nonstandard hull is locally isomorphic to Pestov's nonstandard hull for Banach-Lie groups.  We also give some examples of infinite-dimensional Lie groups that are locally uniform.
\end{abstract}

\section{Introduction}

In \cite{Lux}, Luxembourg constructs the \emph{nonstandard hull} of a uniform space $(X,\u)$.  Roughly speaking, the nonstandard hull of $(X,\u)$ is the quotient of the set of ``finite'' elements of $X^*$ by the equivalence relation of being infinitely close, where $x,y\in X^*$ are infinitely close if the pair $(x,y)$ belongs to the (nonstandard extension of) every entourage in $\u$.  

A natural example of a uniform space is the case of a topological group equipped with either its left uniformity or right uniformity.  A natural question to ask is whether the nonstandard hull of a topological group is naturally a topological group.  In this paper, we show that if the topological group $G$ is \emph{locally uniform}, that is, that the group multiplication is uniformly continuous near the identity, then there is a sutiable modification of Luxembourg's construction that yields a nonstandard hull that is a \emph{local group}.  (A local group is like a topological group except elements can only be multiplied if they are sufficiently close to the identity, and a suitable version of the associative law is required; see \cite{Gold} for a precise definition.)  If one is unhappy about the fact that the nonstandard hull of a topological group is no longer a group, we show that there is a topological group naturally associated to the local group nonstandard hull via the \emph{Mal'cev hull} construction.  Unfortunately, this process does not lead to a canonical choice of \emph{global nonstandard hull}.

A defect of Luxembourg's construction is that it is not a \emph{uniform invariant} in the sense that the construction of the nonstandard hull of a uniform space depends on the choice of a generating set of pseudometrics for the uniformity, and changing the set of pseudometrics drastically changes the appearance of the nonstandard hull.  In this paper, we always use a set of \emph{left-invariant} pseudometrics when constructing the nonstandard hull of a group and discuss the effect of changing the set of generating pseudometrics.  In general, using different sets of pseudometrics does not lead to locally isomorphic nonstandard hulls.  However, if $G$ is metrizable, then any two nonstandard hulls constructed by using left-invariant metrics will be locally isomorphic.

The notion of a locally uniform group was first introduced by Enflo in \cite{enflo} as a way to approach Hilbert's fifth problem in infinite dimensions.  Indeed, locally compact groups are locally uniform and Enflo's aim was to generalize some of the theory of locally compact groups to the setting of locally uniform groups.  In this vein, Enflo introduced a uniform version of the \emph{no small subgroups property} (the property that was integral in settling Hilbert's fifth problem), aptly named \emph{uniformly free from small subgroups}.  The prime examples of groups that are uniformly free from small subgroups are Banach-Lie groups and diffeomorphism groups of compact manifolds.  Groups that are uniformly free from small subgroups are locally uniform and metrizable.  We show that a group is uniformly free from small subgroups if and only if its metric nonstandard hull is free from small subgroups.  (This is a common phenomenon in nonstandard analysis, namely that a standard object has the uniform version of a property if and only if some associated nonstandard object has the non-uniform version of the property.)

Since Banach-Lie groups are locally uniform, our nonstandard hull procedure applies to them.  In \cite{pestov}, Pestov, using a different construction and some nontrivial Lie theory, developed a nonstandard hull construction for Banach-Lie groups.  We will show that, for Banach-Lie groups, our nonstandard hull is locally isomorphic to Pestov's nonstandard hull; in fact, for a suitable choice in our construction, our global nonstandard hull is the universal covering group of Pestov's nonstandard hull.  Our nonstandard hull has the advantage of being a purely topological construction, involving no Lie theoretic facts in its construction.  

Pestov used his nonstandard hull construction to prove a useful local criterion for when a Banach-Lie algebra is enlargeable in the sense that it is the Lie algebra of a Banach-Lie group.  It is our hope that our general nonstandard hull construction will be of use in settling some of the open problems in infinite-dimensional Lie theory presented in \cite{neeb}.  Of course, in order to achieve this goal, it will be useful to understand which infinite-dimensional Lie groups are locally uniform; we devote some time here to discussing this issue.

We assume that the reader is familiar with nonstandard analysis; otherwise the reader can consult \cite{D} or \cite{He}.  We will also assume that the reader is familiar with some basic facts from Lie theory, although we might occasionally recall some of the relevant facts.

We would like to thank Lou van den Dries and Terence Tao for some useful discussions regarding this paper.

\section{Locally Uniform Groups}

If $X$ is a set and $\u$ is a uniformity on $X$, then we set $$\mu(\u):=\bigcap \{U^* \ | \ U\in \u\},$$ and we write $x\approx_\u y$ to indicate $(x,y)\in \mu(\u)$.  If $\u_1$ and $\u_2$ are uniformities on $X$, then it is easy to see that $\u_1=\u_2$ if and only if $\mu(\u_1)=\mu(\u_2)$.  Indeed suppose that $U\in \u_1\setminus \u_2$.  Then for every $V\in \u_2$, $V\setminus U\not=\emptyset$.  By saturation, $\bigcap\{V^* \ V\in \u_2\}\setminus U^*\not=\emptyset$, contradicting that $\mu(\u_1)=\mu(\u_2)$.  It is also easy to see that if $(X,\u)$ and $(Y,\mathcal{V})$ are uniform spaces, then a map $f:X\to Y$ is uniformly continuous if and only if, for all $x,y\in X^*$, $x\approx_\u y\Rightarrow f(x)\approx_{\mathcal{V}} f(y)$.

Throughout this paper, $G$ denotes a (hausdorff) topological group with nonstandard extension $G^*$; we denote the monad of the identity simply by $\mu$.  We let $\u_l$ denote the \emph{left uniformity of $G$}, that is, the uniformity on $G$ which has sets of the form $\{(x,y) \ : \ x^{-1}y\in U\}$ as a basis, where $U$ ranges over the open neighborhoods of the identity.  For $x,y\in G^*$, we write $x\approx_l y$ if and only if $(x,y)\in \mu(\u_l)$; equivalently, $x\approx_l y$ if and only if $x^{-1}y\in \mu$.  Similarly, we have the \emph{right uniformity} $\u_r$ of $G$, which has as a basis sets of the form $\{(x,y) \  \ xy^{-1}\in U\}$, where $U$ ranges over the open neighborhoods of the identity.  For $x,y\in G^*$, we write $x\approx_r y$ if and only if $(x,y)\in \mu(\u_r)$, or, equivalently, $xy^{-1}\in \mu$.  Clearly $\approx_l$ and $\approx_r$ are equivalence relations on $G^*$.

For $A\subseteq G$ and $n\in \n^{>0}$, we write $A^n:=\{x_1\ldots x_n \ : \ \text{ each }x_i\in A\}$.   

\begin{lemma}[\cite{enflo}, Proposition 1.1.2]\label{locuniform}
Suppose that $U$ is a symmetric open neighborhood of the identity and $\u$ is a uniformity on $G$ compatible with the topology on $G$ such that the map $(x,y)\mapsto xy:U^2\times U^2\to U^4$ is $\u$-uniformly continuous.  Then $\u|U=\u_l|U=\u_r|U$ and $x\mapsto x^{-1}:U\to U$ is $\u$-uniformly continuous.
\end{lemma}

\begin{proof}
Suppose that $x,y\in U^*$ are such that $x\approx_\u y$.  Since $x^{-1}\approx_\u x^{-1}$, we get that $x^{-1}x\approx_\u x^{-1}y$, that is, $x^{-1}y\approx_\u e$.  Since $\u$ is compatible with the topology of $G$, we have $x^{-1}y\in \mu$, that is, $x\approx_l y$.  Conversely, suppose that $x\approx_l y$, that is, $x^{-1}y\in \mu$.  Since $\u$ is compatible with the topology of $G$, we have $x^{-1}y\approx_\u e$, so $x\approx_\u y$.  Consequently, $\mu(\u)\cap (U^*\times U^*)=\mu(\u_l)\cap (U^*\cap U^*)$, whence $\u|U=\u_l|U$.  One argues in the same way to obtain the same result for $\u_r$.

Now suppose that $x,y\in U^*$ and $x\approx_\u y$.  Then $x\approx_l y$, so $x^{-1}y\in \mu$, so $x^{-1}y\approx_\u e$, so $x^{-1}yy^{-1}\approx_\u y^{-1}$, that is, $x^{-1}\approx_\u y^{-1}$.     
\end{proof}

Following Enflo, we say that $G$ is \emph{locally uniform} if there is a uniformity $\u$ on $G$ compatible with the topology and a symmetric open neighborhood $U$ of the identity such that the map $(x,y)\mapsto xy:U^2\times U^2\to U^4$ is uniformly continuous.  If we want to specify $U$, we say that $G$ is \emph{$U$-locally uniform}.  If we can take $U=G$, we say that $G$ is \emph{uniform}.

Let $P_l$ denote the set of left-invariant pseudometrics on $G$ which are continuous (as maps from $G\times G$ into $\r$) and let $P_r$ denote the set of right-invariant pseudo-metrics on $G$ which are continuous.  Then $P_l$ generates $\u_l$, that is, the sets of the form $V_{p,r}:=\{(x,y)\in G\times G \ : \ p(x,y)<r\}$ form a subbase for $\u_l$ as $p$ ranges over $P_l$ and $r$ ranges over $\r^{>0}$.  Similarly, $P_r$ generates $\u_r$.  Observe that if $p_1,\ldots,p_n\in P_l$, then $\max(p_1,\ldots,p_n)\in P_l$, whence the sets $V_{p,r}$ form a base for $\u_l$; a similar observation holds for $P_r$ and $\u_r$.
\begin{lemma}
The following are equivalent:
\begin{enumerate}
\item $G$ is $U$-locally uniform;
\item $\u_l|U=\u_r|U$;
\item $\mu(\u_l)\cap (U^*\times U^*)=\mu(\u_r)\cap (U^*\times U^*)$;
\item for all $x,y\in U^*$:  $x^{-1}y\in \mu \Leftrightarrow xy^{-1}\in \mu$;
\item $\mu$ is ``normal'' in $U^*$:  for all $x\in U^*$ and $y\in \mu$, we have $xyx^{-1}\in \mu$;
\item for all $x,y\in U^*$: $$p(x,y)\approx 0 \text{ for all }p\in P_l \Leftrightarrow q(x,y)\approx 0 \text{ for all }q\in P_r.$$
\end{enumerate}
\end{lemma}

\begin{proof}
The direction (1)$\Rightarrow (2)$ was part of Lemma \ref{locuniform} and clearly $(2)$-$(6)$ are equivalent.  We prove (3)$\Rightarrow (1)$.  Fix $x,x_1,y,y_1\in U^*$ such that $x\approx_l x_1$ and $y\approx_l y_1$.  It suffices to show that $xy\approx_l x_1y_1$.  It is clear that $xy\approx_lxy_1$ and, by (3), we have $xy_1\approx_r x_1y_1$.  By (3) again, we have $xy_1\approx_l x_1y_1$, whence $xy\approx_l x_1y_1$.
\end{proof}

\begin{exs}

\

\begin{enumerate}
\item Any locally compact group is locally uniform; in fact, if $U$ is a symmetric open neighborhood of the identity with compact closure, then $G$ is $U$-locally uniform.  In particular, compact groups are uniform.
\item Any abelian group is uniform.  More generally, if $U$ is a symmetric open neighborhood of the identity such that $xy=yx$ for all $x,y\in U$, then $G$ is $U$-locally uniform.
\item If $G$ admits a two-sided invariant metric, then $G$ is uniform.  More generally, if $d$ is a metric for $G$ and $U$ is a symmetric open neighborhood of the identity such that $d|(U^2\times U^2)$ is two-sided invariant, then $G$ is $U$-locally uniform.
\end{enumerate}
\end{exs}

\noindent We will encounter other examples of locally uniform groups later.  The following characterization of uniform groups appears in Enflo (without proof) and has an easy nonstandard proof.

\begin{lemma}
$G$ is uniform if and only if for every neighborhood $U$ of $e$ there is a neighborhood $V$ of $e$ such that $gVg^{-1}\subseteq U$ for all $g\in G$.
\end{lemma}

\begin{proof}
Suppose that $G$ is uniform and $U$ is a neighborhood of $e$.  Let $V\subseteq \mu$ be an internal neighborhood of the identity.  For $g\in G$, we have $gVg^{-1}\subseteq \mu\subseteq U^*$ by (5) of the above theorem.  By transfer, the desired $V$ exists.  For the converse, we prove that $\mu$ is normal in $G^*$.  Given $x\in G^*$ and $y\in \mu$, we must show that $xyx^{-1}\in \mu$.  Fix $U$ an open neighborhood of the identity.  Then we are guaranteed $V$ so that $gVg^{-1}\subseteq U$ for all $g\in G$.  By transfer, we have that $xyx^{-1}\in x\mu x^{-1}\subseteq xV^*x^{-1}\subseteq U^*$.  Thus, $xyx^{-1}\in \mu$.
\end{proof}

\section{Nonstandard hulls}
Generalizing the notion of a nonstandard hull of a normed space, Luxembourg \cite{Lux} constructs a nonstandard hull for any uniform space $(X,\u)$ as follows.  Fix a set $P$ of pseudometrics generating $\u$.  Set $$X_{f,P}:=\{x\in X^* \ | \ p(x)\in \r_{f} \text{ for all }p\in P\}.$$  For $x\in X_{f,P}$, let $[x]$ denote the equivalence class of $x$ with respect to the equivalence relation $\approx_\u$.  Set $\hat{X}_P:=\{[x]\ | \ x\in X_{f,P}\}$.  Then $\hat{X}_P$ is a uniform space with respect to the family of pseudometrics $^\circ P:=\{^\circ p \ | \ p\in P\}$, where $^\circ p([x]):=\st(p(x))$.

Ideally, one would hope that the nonstandard hull of a topological group could once again be equipped with a group structure such that the resulting group is a topological group.  However, showing that the infinitesimals are a normal subgroup of the finite elements requires that the group multiplication be uniformly continuous.  If we only assume that the topological group is locally uniform, then we can obtain a nonstandard hull which is a \emph{local group}; we refer the reader to \cite{Gold} for an introduction to local groups.  

Let us carry out these details now.  First, for $x\in G^*$, we set $$\mu(x):=\{y\in G^* \ : \ x\approx_l y\} \quad \text{ and }\i(U^*):=\{x\in U^* \ | \ \mu(x)\subseteq U^*\}.$$  We clearly have that $U\subseteq \i(U^*)$.

Suppose that $G$ is $U$-locally uniform and $S\subseteq P_l$ is such that $S$ generates $\u_l$.  Without loss of generality, we may assume that if $p_1,\ldots,p_n\in S$, then $\max(p_1,\ldots,p_n)\in S$.  

We set $$U_{S,f}:=\{x\in \i(U^*) \ : \ p(x,e)\in \r_f \text{ for all }p\in S\}.$$  Given $x\in U_{S,f}$, we write $[x]$ for the $\approx_l$-equivalence class of $x$  (which coincides with the $\approx_r$-equivalence class of $x$).  We then set $\hat{U}_S:=\{[x] \ : \ x\in U_{S,f}\}$.  Given $p\in S$, we define $\sta p:\hat{U}_S\to \r$ by $\sta p([x]):=\st(p(x))$.  Then $\sta p$ is a left-invariant pseudometric on $\hat{U}_S$.  We view $\hat{U}_S$ as a uniform space by giving it the uniformity generated by the set $\hat{S}:=\{\sta p \ : \ p\in S\}$.  

Notice that if $x,y,x_1,y_1\in U_{S,f}$ are such that $x\approx_l x_1$, $y\approx_l y_1$, and $xy\in U_{S,f}$, then $x_1y_1\in U_{S,f}$.  Thus, we can set $$\o:=\{([x],[y]) \in \hat{U}_S\times \hat{U}_S \ : \ xy\in U_{S,f}\}.$$  We first claim that $\o$ is open.  (In fact, this was the entire reason for requiring $U_{S,f}\subseteq \i(U^*)$.)  Fix $([x],[y])\in\o$.  By saturation, there is $p\in S$ and $r\in \r^{>0}$ such that, for all $z\in G^*$, if $p(xy,z)<r$, then $\mu(z)\subseteq U^*$.  By uniform continuity, there is $p'\in S$ and $r'\in \r^{>0}$ such that, for all $a,b,c,d\in U^2$, if $p'(a,c),p'(b,d)<r'$, then $p(ab,cd)<r$.  Now suppose that $([x_1],[y_1])\in \hat{U}_S\times \hat{U}_S$ is such that $\sta p'([x],[x_1]),\sta p'([y],[y_1])<r'$.  Then $p'(x,x_1),p'(y,y_1)<r'$, so $p(xy,x_1y_1)<r$, whence $\mu(x_1y_1)\subseteq U^*$.  If $q\in S$, then $q(x_1y_1,e)\leq q(x_1y_1,x_1)+q(x_1,e)=q(y_1,e)+q(x_1,e)\in \r_f$.  Consequently, $\mu(x_1y_1)\subseteq U_{S,f}$.  It follows that $\o$ is open.

By uniform continuity, we can define $m:\o \to \hat{U}_S$ by $m([x],[y]):=[xy]$.  Arguing as in the previous paragraph, we see that $m$ is uniformly continuous.  Next notice that if $x\in U_{S,f}$, then $x^{-1}\in U_{S,f}$.  Indeed, it is easy to see that $x^{-1}\in U^*$ and $p(x^{-1},e)\in \r_f$ for every $p\in S$.  It remains to see that $\mu(x^{-1})\subseteq U^*$.  However, by uniform continuity, we have that $\mu(x^{-1})=\mu(x)^{-1}\subseteq U^*$ because $x\in \i(U^*)$ and $U$ is symmetric.  Thus, by uniform continuity, we can define $\iota:\hat{U}_S\to \hat{U}_S$ by $\iota([x])=[x^{-1}]$.  It is easy to see that $\iota$ is continuous.  It follows that $(\hat{U}_S,m,\iota,[e])$ is a globally inversional local group.  Moreover, $U\subseteq U_{S,f}$ and  the map $x\mapsto [x]:U\to \hat{U}_S$ is a uniformly continuous, injective strong morphism of local groups.

\begin{rmk}
Suppose that $A\subseteq U^*$ is such that:
\begin{itemize}
\item for all $x\in A$, $x^{-1}\in A$;
\item there is an open $V\subseteq U$ such that, for all $(x,y)\in A\times A$, if $xy\in V^*$, then $xy\in A$.
\end{itemize}

\noindent Then $\hat{A}:=\{[x] \ | \ x\in A\}$ is a local subgroup of $\hat{U}_S$.  Indeed, let $p\in S$ be such that $\{x\in G \ | \ p(x,e)<1\}\subseteq V$.  Now suppose that $([x],[y])\in (\hat{A}\times \hat{A})\cap \o$ and $\sta p([xy],[e])<\frac{1}{2}$.  Then $p(xy,e)<1$, so $xy\in V^*$, whence $xy\in A$.  Thus, $[x]\cdot [y]\in \hat{A}$.
\end{rmk}

\begin{rmk}
One should note that in the case that $G$ is uniform, the nonstandard hull we just constructed is just the usual nonstandard hull of a uniform space as constructed by Luxembourg.
\end{rmk}

\begin{ex}\label{compact}
Suppose that $G$ is locally compact.  Further suppose that $U$ is a symmetric open neighborhood of the identity with compact closure.  Then it is easy to see that $U_{S,f}=\i(U^*)=\{x\in U^* \ : \ x\approx_l y\text{ for some }y\in U\}$.  Consequently, the map $x\mapsto [x]:U\to \hat{U}_S$ is an isomorphism of local groups.
\end{ex}

\begin{ex}
If $S=P_l$, then we write $U_{f}$ instead of $U_{P_l,f}$ and $\hat{U}$ instead of $\hat{U}_{P_l}$.  We refer to $\hat{U}$ as the \emph{canonical nonstandard hull}.  Observe that the canonical nonstandard hull is in some sense the ``smallest'' of the nonstandard hulls.  Indeed, if $S\subseteq P_l$ is as above, then $U_f\subseteq U_{S,f}$ and the mapping $[x]\mapsto [x]:\hat{U}\to \hat{U}_S$ is an injective morphism of local groups.
\end{ex}

\begin{ex}\label{allfinite}
At the opposite extreme, if we set $P_1:=\{\min(p,1) \ : \ p\in P_l\}$, then $P_1$ generates $\u_l$ and $U_{P_1,f}:=\i(U^*)$.  Consequently, if $S\subseteq P_l$ is any generating set of pseudometrics, then the map $[x]\mapsto [x]:\hat{U}_S \to\hat{U}_{P_1}$ is an injective morphism of local groups.
\end{ex}

If $H$ is a local group with domain of multiplication $\Omega_H$, then one defines the notion ``$x^n$ is defined'' by recursion on $n$:  $x^1$ is always defined and, for $n\geq 2$, $x^n$ is defined if $x^i$ is defined for all $i<n$ and $(x^i,x^j)\in \Omega_H$ for all $i,j<n$ such that $i+j=n$.  (This is not the definition given in \cite{Gold} but is proved to be equivalent there.)
\begin{rmk}\label{defined}
An easy inductive argument shows that, for all $x\in U_{S,f}$, if $[x]^n$ is defined, then $x^n\in U_{S,f}$ and $[x]^n=[x^n]$.
\end{rmk}

\begin{rmk}
Observe that we only constructed nonstandard hulls for locally uniform groups.  In fact, for metrizable groups, this was a necessary assumption.  Indeed, suppose that $G$ is a topological group and we wanted to define $[x]\cdot [y]:=[xy]$ for $x,y\in U_F$, where $U_F$ is the set of all elements of $U^*$ ($U$ a neighborhood of the identity) which are ``finite'' in some sense.  Any sensible notion of finiteness will include the requirement that $V^*\subseteq U_F$ for some neighborhood $V$ of the identity.  Then the well-definedness of the group operation on the nonstandard hull implies that multiplication on $U_F$ be $S$-continuous, whence multiplication on $V$ is uniformly continuous, whence $G$ is locally uniform.
\end{rmk}

\begin{lemma}
Suppose that $G$ is both $U$-locally uniform and $V$-locally uniform.  Then $\hat{U}_S$ and $\hat{V}_S$ are locally isomorphic local groups.
\end{lemma}

\begin{proof}
Let $p\in S$ and $r\in \r^{>0}$ be such that $V_{p,r}^*\subseteq \i(U^*)\cap \i(V^*)$.  Then $[x]\mapsto[x]:\hat{U}|V_{\sta p,r}\to \hat{V}|V_{\sta p,r}$ is a local group isomorphism.
\end{proof}

In view of the preceding lemma, given a set $S$ of generating pseudometrics, all of the above nonstandard hulls are locally isomorphic to one another and thus are ``essentially the same'' local group.  In fact, we could even define the \emph{germ nonstandard hull of $G$ with respect to $S$} to be the local group germ made up of all nonstandard hulls $\hat{U}_S$, where $G$ is $U$-locally uniform.  (A \emph{local group germ} is the equivalence class of a local group where the equivalence relation is local isomorphism.)

Sometimes, changing the generating set of pseudometrics has no effect on the nonstandard hull.

\begin{lemma}
Suppose that $U$ is metrizable, with compatible (not necessarily left-invariant) metric $d$.  Further suppose that there is a continuous map $\sqrt{\cdot}:U\to U$ so that $\sqrt{x}\cdot \sqrt{x}=x$ for all $x\in U$ and so that $\sqrt{\cdot}$ is a $d$-contraction.  Then $p(x,e)\in \r_f$ for all $p\in P$ and all $x\in U^*$ with $d(x,e)\in \r_f$.
\end{lemma}

\begin{proof}
Given $p\in P$, there is $\epsilon>0$ so that $d(y,e)<\epsilon$ implies $p(y,e)<1$.  Fix $c\in (0,1)$ such that $d(\sqrt{x},\sqrt{y})\leq c\cdot d(x,y)$ for all $x,y\in U$.  Given $n\in \n$ and $x\in U$, let $x^{\frac{1}{2^n}}$ denote the $n^{\text{th}}$ iterate of $\sqrt{\cdot}$ applied to $x$.  Then, for any $n$, $d(x^{\frac{1}{2^n}},e)\leq c^nd(x,e)$.  If $x\in U^*$ is such that $d(x,e)\in \r_f$, then for $n$ sufficiently large, $d(x^{\frac{1}{2^n}},e)<\epsilon$, whence $p(x^{\frac{1}{2^n}},e)<1$.  Now notice that $p(x,e)\leq p(x,\sqrt{x})+p(\sqrt{x},e)=2p(\sqrt{x},e)$ by left-invariance.  By induction, $p(x,e)\leq 2^np(x^{\frac{1}{2^n}},e)<2^n$.
\end{proof}

\begin{rmk}
In the proof of the previous lemma, all we needed was there to be $c\in (0,1)$ such that $d(\sqrt{x},e)\leq c\cdot d(x,e)$ for all $x\in U$.
\end{rmk}

\begin{cor}
Suppose that $E$ is a normed vector space and $\exp:E\to G$ is a local homeomorphism, say $\exp|V:V\to U$ is a homeomorphism, where $V$ is a balanced neighborhood of $0$.  Then there is an open neighborhood $U_1$ of $e$ in $G$ contained in $U$ such that $p(x,e)\in \r_f$ for all $x\in U_1^*$ and all $p\in P$.
\end{cor}

\begin{proof}
Let $d$ be the metric on $U$ given by $d(\exp(x),\exp(y)):=\|x-y\|$.  Now define $\sqrt{\cdot}:U\to U$ by $\sqrt{\exp(x)}:=\exp(\frac{1}{2}{x})$.  Then, for all $\exp(x)\in U$, we have $d(\sqrt{\exp(x)},e)=\|\frac{1}{2}x\|=\frac{1}{2}\|x\|=\frac{1}{2}d(\exp(x),e)$.  Thus, by the previous lemma, for all $x\in U^*$ with $d(x,e)\in \r_f$, we have $p(x,e)\in \r_f$ for all $p\in P$.  Thus, we can take $U_1$ to be any $d$-ball centered at $0$ of a finite radius.
\end{proof}

In particular, by the previous result, if $G$ is a locally exponential Lie group whose Lie algebra is normable, then we can form the nonstandard hull using any collection of left-invariant pseudometrics that generates the uniformity $\u_l$.   

\

\noindent \textbf{Metric nonstandard hulls}

\

Suppose now that $G$ is metrizable, say with left-invariant metric $d$.  We can then take $S=\{d\}$, in which case we denote $U_{S,f}$ by $U_{d,f}$ and $\hat{U}_S$ by $\hat{U}_d$.  Let $r\in \r^{>0}$ be such that, setting $U:=B_d(e,r)$, we have that $G$ is $U$-locally uniform.  Then $U_{d,f}:=\{x\in U^* \ | \ \st(d(x,e))<r\}$ and the local group $\hat{U}_d$ is metrizable with metric $\hat{d}([x],[y]):=\st(d(x,y))$.  We now address the question:  How are various metric nonstandard hulls related?

Suppose that $d$ is a left-invariant metric on $G$ and $S\subseteq P_l$ is a generating set of pseudometrics closed under $\max$.  Let $s\in \r^{>0}$ be small enough so that, setting $V:=B_d(e,s)$, we have that $G$ is $V$-locally uniform.  Fix $p\in S$ and $r\in \r^{>0}$ so that, for all $x\in G$, if $p(x,e)<r$, then $d(x,e)<\frac{s}{2}$.  Set $U:=\{x\in G \ | \ p(x,e)<r\}$, so $G$ is also $U$-locally uniform.  Notice also that if $x\in U_{P,f}$, then $\st(d(x,e))<s$, whence $x\in V_{d,f}$.  Consequently, we get a map $\phi:\hat{U}_P\to \hat{V}_d$, $\phi([x]_P)=[x]_d$, which is clearly injective and a morphism of discrete local groups.  In order to see that $\phi$ is continuous, it suffices to check continuity at $[e]_P$.  Given $\epsilon>0$, choose $q\in P$ and $\delta>0$ so that, for all $x\in G$, if $q(x,e)<\delta$, then $d(x,e)<\frac{\epsilon}{2}$.  Now suppose that $\hat{q}([x]_P,[e]_P)<\delta$.  Then $q(x,e)<\delta$, so $d(x,e)<\frac{\epsilon}{2}$, whence $\hat{d}([x]_d,[e]_d)<\epsilon$.  It follows that the metric nonstandard hull is, in some sense, the largest nonstandard hull.

Now suppose, in addition, that $S=\{d'\}$, where $d'$ is also a left-invariant metric on $G$.  We claim now that the map $\phi$ (where $U=B_{d'}(e,r)$) is an open morphism of local groups with open image.  Indeed, take $\epsilon\in (0,r]$.  We must show that $\phi(B_{\hat{d'}}([e]_{d'},\epsilon))$ is an open subset of $\hat{V}_d$.  Suppose that $\hat{d'}([x]_{d'},[e]_{d'})<\epsilon$, so $\st(d'(x,e))<\epsilon$.  Take $\delta\in \r^{>0}$ such that, for all $y\in G^*$, if $d'(x,y)<\delta$, then $\st(d'(y,e))<\epsilon$.  Take $\eta\in \r^{>0}$ such that, for all $a,b\in G$, if $d(a,b)<\eta$, then $d'(a,b)<\delta$.  (This uses left-invariance of both $\delta$ and $\delta'$.)  Now suppose that $\hat{d}([x]_d,[y]_d)<\eta$.  Then $d'(x,y)<\delta$, so $\st(d'(y,e))<\epsilon$.  In other words, $B_{\hat{d}}([x]_d,\eta)\subseteq \phi(B_{\hat{d'}}([e]_{d'},\epsilon))$.

It now follows that $\hat{U}_{d'}$ is isomorphic to $\hat{V}_d|\phi(\hat{U}_{d'})$.  Indeed, suppose that $[x]_d,[y]_d\in \phi(\hat{U}_{d'})$ are such that $([x]_d,[y]_d)\in \Omega_{\hat{V}_d}$ and $[x]_d\cdot [y]_d\in \phi(\hat{U}_{d'})$.  Then $[xy]_d=[z]_d$ for some $z\in U_{d',f}$.  Since $xy\approx z$, it follows that $xy\in U_{d',f}$, so $([x],[y])\in \Omega_{\hat{U}_{d'}}$.  In particular, we have proven the following result:

\begin{prop}
Any two metric nonstandard hulls are locally isomorphic.
\end{prop}

Since any two metric nonstandard hulls are locally isomorphic, we can speak of the \emph{metric germ nonstandard hull}.



\

\noindent \textbf{Global nonstandard hulls}

\

One may be a bit perturbed by the fact that the nonstandard hull of a locally uniform group is merely a local group.  However, some local groups (most importantly for us, our local nonstandard hulls) can be embedded into topological groups, a procedure that we briefly recall here; more details can be found in \cite{vdd-Gold}.  Until further notice, we let $H$ denote a globally inversional local group.

First, there exists a topological group $H^M$, called the \emph{Mal'cev hull of $H$}, and a local group morphism $\iota:H\to H^M$, satisfying the universal property that whenever $\phi:H\to T$ is a local group morphism into a topological group, there is a unique topological group morphism $\varphi:H^M\to T$ such that $\phi=\varphi \circ \iota$.  In fact, $H^M$ is the set of words on $H$ modulo the equivalence relation generated by the following four operations:

\begin{itemize}
\item $(x_1,\ldots,x_m)\to (x_1,\ldots,x_{i-1},x_ix_{i+1},x_{i+2},\ldots,x_m)$ if $(x_i,x_{i+1})\in \Omega_H$.
\item $(x_1,\ldots,x_m)\to (x_1,\ldots,x_{i-1},x_{i+2},\ldots,x_m)$ if $(x_i,x_{i+1})\in \Omega_H$ and $x_ix_{i+1}=1$.
\item $(x_1,\ldots,x_m)\to (x_1,\ldots,x_{i-1},a,b,x_{i+1},\ldots,x_m)$ if $x_i=ab$ with $(a,b)\in \Omega_H$.
\item $(x_1,\ldots,x_m)\to (x_1,\ldots,x_i,a,a^{-1},x_{i+1},\ldots,x_m)$ for any $a\in H$.
\end{itemize}

If we let $w_M$ denote the equivalence class of the word $w$ in $H^M$, then the group operation on $H^M$ is $w_M\cdot w'_M:=(w ^\frown w')_M$, the map $\iota:H\to H^M$ is given by $\iota(x)=(x)_M$, and the map $\varphi:H^M\to T$ is given by $\varphi((x_1,\ldots,x_m)_M)=\varphi(x_1)\cdots \varphi(x_m)$.

Given elements $a_1,\ldots,a_n\in H$ and $b\in H$, we write $(a_1,\ldots,a_n)\leadsto b$ to mean that there is a way of introducing parentheses into the sequence $a_1,\ldots,a_n$ such that all intermediate products exist and the resulting overall product is $b$.  (See \cite{vdd-Gold} for a precise definition by recursion.)
\begin{df}

\

\begin{enumerate}
\item We say that $H$ is \emph{neat} if $(x,y)\in \Omega_H$ implies $(xy,y^{-1})\in \Omega_H$.
\item We say that $H$ is $\infty$-associative if whenever $a_1,\ldots,a_n,b,c\in H$ are such that $(a_1,\ldots, a_n)\leadsto b$ and $(a_1,\ldots,a_n)\leadsto c$, then $b=c$.
\end{enumerate}
\end{df}

\begin{fact}[Mal'cev \cite{M}; van den-Dries \& Goldbring \cite{vdd-Gold}]
If $H$ is neat and $\infty$-associative, then $\iota:H\to H^M$ is \emph{injective}.
\end{fact}

We now return to the setting of a locally uniform group $G$ and a neighborhood $U$ of $1$ in $G$ such that $G$ is $U$-locally uniform.  We fix a generating set $S\subseteq P_l$ of pseudometrics and suppress mention of $S$ for the remainder of this subsection.  It is then easy to see that the local group $\hat{U}$ is $\infty$-associative and neat.  Consequently, $\hat{U}$ embeds into its Mal'cev hull, which we denote by $\hat{G}_U$.  One may refer to $\hat{U}$ as a \emph{local nonstandard hull} of $G$ and to $\hat{G}_U$ as a \emph{global nonstandard hull of $G$}.  

One may wonder what the relationship is between $G$ and $\hat{G}_U$?  First recall that the map $x\mapsto [x]:U\to \hat{U}$ is an injective morphism of local groups.  Since $\hat{U}$ embeds into its Mal'cev hull $\hat{G}_U$, we have an injective morphism of local groups $\phi:U\to \hat{G}_U$.  Let $U^M$ denote the Mal'cev hull of $U$ and let $\varphi:U^M\to \hat{G}_U$ be the unique topological group morphism ``extending'' $\phi$.  

\begin{lemma}
With the notation as above, we have that $\varphi$ is injective.  
\end{lemma}

\begin{proof}
Suppose that $\varphi((x_1,\ldots,x_m)_M)$ is the identity.  Then $\phi(x_1)\cdots \phi(x_m)$ is the identity in $\hat{G}_U$, meaning that there is a sequence $w_1,\ldots,w_k$ of words on $\hat{U}$ starting with $([x_1],\ldots,[x_m])$ and ending in the empty word, where each $w_{i+1}$ is obtained from $w_i$ using one of the four ``moves'' from above.  We then have that $(x_1,\ldots,x_m)$ is internally equivalent to the empty word in $U^M$, whence by transfer, it is actually equivalent to the empty word.
\end{proof}

By the universal mapping property, there is a canonical group morphism $i^M:U^M\to G$ induced by the inclusion $i:U\to G$.  Since $i^M$ acts homeomorphically on $U$, $i^M$ is a covering map.  If $i^M$ is an isomorphism, then we get an injective morphism $\varphi:G\to \hat{G}_U$ of topological groups (which is something that one expects of a nonstandard hull operation).  There is one natural setting when $i^M:U^M\to G$ is an isomorphism, as the following unpublished result of Lou van den Dries demonstrates:

\begin{prop}[van den Dries]\label{malsim} Suppose $G$ is locally path-connected and simply connected, and $U$ is connected. Then $i^M$ is a topological group isomorphism. 
\end{prop}
\begin{proof} Since $U$ is connected, so is $U^M$ (as the image of $U$ generates $U^M$).  It remains to use the fact (which is a standard consequence of the Monodromy Theorem) that  any covering map $Y\to X$, where $Y$ is connected and $X$ is simply connected and locally path-connected, is an isomorphism.
\end{proof}


For use in the last section, we will need the following other unpublished result of Lou van den Dries concerning the map $i^M$.

\begin{prop}[van den Dries]\label{unimal} Suppose $G$ is connected and locally simply connected.  Let $V$ be a simply connected open neighborhood 
of $1$ 
in $G$, and $U$ a connected symmetric open neighborhood of $1$ in 
$G$ with $U^2\subseteq V$. Then $i^M: U^M \to G$ is a 
universal group covering of $G$.
\end{prop}
\begin{proof} We only need to show that
$U^M$ is simply connected.
Let $p: \tilde{G} \to G$ be the usual universal group covering 
of $G$, let $V'$ be the connected component of the identity in $p^{-1}(V)$. 
Then $p$ maps $V'$ homeomorphically onto $V$, and
$V'$ is open-and-closed in $p^{-1}(V)$. Let $v\mapsto v': V \to V'$ be the
inverse of $x\mapsto p(x): V' \to V$, and set 
$$U'\ :\ =\{u': u\in U\}\ =\ V'\cap p^{-1}(U).$$ 
Then $U'$ is a connected open neighborhood of the identity
in $\tilde{G}$. Moreover, $U'^2\subseteq V'$, because 
$u_1'u_2'\in p^{-1}(U^2)\subseteq p^{-1}(V)$ for all $u_1, u_2\in U$
and so $$\{(u_1,u_2)\in U\times U:\ u_1'u_2'\in V'\}$$ is
open-and-closed in the connected space $U\times U$. Likewise, considering
$\{u\in U:\ u'^{-1}\in V'\}$ we see
that $U'$ is symmetric. It is now easy to check 
that $u \mapsto u' \colon G|U \to \tilde{G}|U'$ is an isomorphism of 
local groups.
This induces a topological group isomorphism $U^M\cong (U')^M$.
As $U'$ is connected, $(U')^M$ is
homeomorphic to $\tilde{G}$, by Corollary~\ref{malsim}. 
Thus $U^M$ is simply connected.
\end{proof}

We should stress that neither Proposition \ref{malsim} nor Proposition \ref{unimal} require that $G$ be locally uniform.

In general, if $G$ is both $U$-locally uniform and $V$-locally uniform, then $\hat{G}_U$ and $\hat{G}_V$ can be non-isomorphic, whence the global nonstandard hull construction is \emph{non-canonical}.  For example, let $G$ be a compact, locally connected, non-connected group (e.g. $G=O_n(\r)$).  Let $U$ be the connected component of the identity in $G$, a symmetric open neighborhood of the identity.  Then $G$ is both $U$-locally uniform and $G$-locally uniform.  Notice that $\hat{U}$ is isomorphic (as a local group) to $U$ (see Example \ref{compact}), so $\hat{G}_U$ is connected, while $\hat{G}_G=\hat{G}$ is isomorphic to $G$, which is not connected.

\section{$\u$-finiteness}

There is another notion of finiteness for uniform spaces due to Henson \cite{henson}.  Suppose that $(X,\u)$ is a uniform space.  We say that $a\in X^*$ is \emph{$\u$-finite} if, for every $A\in \u$, there is a sequence $a_0,\ldots,a_n$ from $X^*$ such that $a_0=a$, $a_n\in X$, and $(a_i,a_{i+1})\in A^*$ for each $i<n$.  It is easy to see that every $\u$-finite element of $X$ is also finite in our above sense, that is, if $P$ is a family of pseudometrics generating $\u$, then whenever $a\in X^*$ is $\u$-finite, then $p(a,x)\in \r_f$ for all $p\in P$ and all $x\in X$.  We let $X_f^\u$ denote the set of $\u$-finite points of $X^*$.

Returning to our situation of locally uniform groups:  Suppose that $G$ is $U$-locally uniform and suppose that $V$ is a symmetric open neighborhood of the identity such that $V^2\subseteq U$.

\begin{lemma}
$(U_f^\u)^{-1}=U_f^\u$ and $V_f^\u\cdot V_f^\u\subseteq U_f^\u$.
\end{lemma}

\begin{proof}
We only prove the second assertion; the first assertion is similar using uniform continuity of inversion.  Suppose that $x,y\in V_f^\u$.  Fix $p\in P$ and $\epsilon\in \r^{>0}$.  Take a sequence $y_0,\ldots,y_n\in V^*$ such that $y_0=y$, $y_n\in V$ and $p(y_i,y_{i+1})<\epsilon$ for each $i<n$.  Since the map $a\mapsto ay_n:V\to U$ is uniformly continuous, there is $q\in P$ and $\delta>0$ such that, for all $a,a'\in V$, $q(a,a')<\delta\Rightarrow p(ay_n,a'y_n)<\epsilon$.  Take a sequence $x_0,\ldots,x_m\in V^*$ such that $x_0=x$, $x_m\in V$, and $q(x_j,x_{j+1})<\delta$ for $j<m$.  Then the sequence $x_0y_0,\ldots,x_0y_n,x_1y_n,\ldots,x_my_n$ witnesses that $xy\in U_f^\u$.
\end{proof}

\begin{cor}
If $G$ is a uniform group, then $G_f^\u$ is a subgroup of $G_f$.
\end{cor}

Suppose that $G$ is a uniform group and set $\hat{G}^\u:=G_f^\u/\mu$.  By the last corollary, $\hat{G}^\u$ is a subgroup of $\hat{G}_S$ for any generating set $S\subseteq P_l$.

By Theorems 3.2 and 3.3 of \cite{henson}, we have $X^*=X_f^\u$ if and only if every uniformly continuous function $X\to \r$ is bounded if and only if $X$ is ``finitely chainable.''  (This is some strengthening of the notion of ``pseudocompact.'')  In particular, if $X$ is finitely chainable, then $p(a,x)\in \r_f$ for all $p\in P$, all $a\in X^*$, and all $x\in X$, where $P$ is a generating family of pseudometrics for the uniformity on $X$.  Thus, for finitely chainable uniform groups, it does not matter what family of left-invariant pseudometrics we take in the definition.  We should observe that if a group is finitely chainable, then it can never have $\r$ as a quotient (for the quotient map $\pi:G\to \r$ would be uniformly continuous).  More generally, if $G$ is complete and finitely chainable, then any topological group morphism $\phi:G\to \r$ must be trivial.  (Indeed, in this case, $\phi(G)$ is a closed subgroup of $\r$, whence is $\{0\}$, $\r$, or $\z\cdot r$ for some $r\in \r$.)  Observe also that if $G$ is uniform but not finitely chainable, then $\hat{G}^\u$ is a proper subgroup of the canonical nonstandard hull.

\begin{question}
Is there a characterization of the finitely chainable groups?
\end{question}

\

\noindent \textbf{A test-case:  Locally convex vector spaces}

\

\noindent We now consider the special case of the additive group of a locally convex vector space.  Let $(E,+)$ be a locally convex space.  Then there is yet another notion of finiteness for such spaces.  Let $\Gamma$ denote the set of continuous seminorms on $E$.  We say that $x\in E^*$ is \emph{tvs-finite} if $p(x)\in \r_f$ for every $p\in \Gamma$.  We let $E_f^{\tvs}$ denote the set of tvs-finite elements of $E^*$.

\begin{prop}
$E_f=E_f^\u=E_f^{\tvs}$.
\end{prop}

\begin{proof}
Suppose that $x\in E_f^\u$ and let $p\in \Gamma$.  Take a sequence $x_0,\ldots,x_n\in E^*$ such that $x_0=x$, $x_n\in E$, and $p(x_i-x_{i+1})<1$ for $i<n$.  Then $p(x-x_n)<n$, so $p(x)\in \r_f$.  Thus, $x\in E_f^{\tvs}$ and $E_f^\u\subseteq E_f^{\tvs}$.  Conversely, suppose that $x\in E_f^{\tvs}$.  Fix $p\in \Gamma$ and $\epsilon>0$.  Let $n\in \n^{>0}$ be such that $\frac{1}{n}p(x)<\epsilon$; this is possible because $p(x)\in \r_f$.  Let $x_i:=\frac{n-i}{n}x\in E^*$.  Then $x_0=x$, $x_n=0$, and $p(x_i-x_{i+1})=p(\frac{1}{n}x)=\frac{1}{n}p(x)<\epsilon$.  Thus $x\in E_f^\u$ and $E_f^\u=E_f^{\tvs}$.  It remains to show that $E_f\subseteq E_f^{\tvs}$.  However, this follows from the fact that any $p\in \Gamma$ induces $\hat{p}\in P_l$ by $\hat{p}(x,y)=p(x-y)$.
\end{proof}

Consequently, we get one notion of a nonstandard hull for the additive group of a locally convex space, which we may unambiguously write as $\hat{E}$. 

\section{Functoriality}

We would like the above construction to be functorial, that is, if $G$ and $H$ are locally uniform and $f:G\to H$ is a topological group morphism, we would like to obtain an induced morphism $\hat{f}:\hat{U}\to \hat{V}$.  Fix $S\subseteq P_{l,G}$ and $S'\subseteq P_{l,H}$.

\begin{lemma}
Suppose that $f(U_{S,f})\subseteq V_{S',f}$ and $f(\i(U^*))\subseteq \i(V^*)$.  Then there is an induced map $\hat{f}:\hat{U}_S\to \hat{V}_{S'}$ given by $\hat{f}([x]):=[f(x)]$.  Morever, $\hat{f}$ is a morphism of local groups.
\end{lemma}

\begin{proof}
First suppose that $x,x_1\in \i(U^*)$ are such that $x\approx x_1$.  Then $x^{-1}x_1\approx e$, so $f(x^{-1}x_1)\approx e$ by continuity of $f$, whence $f(x)\approx f(x_1)$.  Consequently, we can define $\hat{f}$ as in the statement of the lemma.  It is clear that $\hat{f}$ respects multiplication and inversion.  To check continuity of $\hat{f}$, it suffices to check continuity at $[e_G]$; however, this follows easily from the continuity of $f$ at $e_G$.  
\end{proof}

We first would like to know how to ensure that $f(U_{S,f})\subseteq V_{S',f}$.  Certainly, if $S'=P_{H,1}$, then this is satisfied.  (See Example \ref{allfinite}.)  Also, if we use $S=P_{l,G}$, then this is satisfied as well.  Indeed, given a continuous left-invariant pseudometric $q$ on $H$, we have that $q\circ (f\times f)$ is a continuous left-invariant pseudometric on $G$.  Moreover, if $x\in G_{S,f}$, then $q(f(x),e_H)=q(f(x),f(e_G))=(q\circ (f\times f))(x,e_G)\in \r_f$.  Finally, if $G$ and $H$ are metrizable, say with left-invariant metrics $d$ and $d'$, and $U$ and $V$ are suitable open balls, then $f(U_{d,f})\subseteq V_{d',f}$.

The more serious issue is how to ensure that $f(\i(U^*))\subseteq \i(V^*)$.  The easiest case to deal with is the metric case.  Indeed, suppose that $H$ is metrizable and $f(U)\subseteq B(e_H,\frac{\epsilon}{2})$.  Then $f(U^*)\subseteq B(e_H,\frac{\epsilon}{2})^*\subseteq \i(B(e_H,\epsilon)^*)$.  We have thus established:

\begin{prop}
Let $\operatorname{MGrp}$ denote the category of metrizable topological groups with continuous group morphisms as arrows.  Let $\operatorname{LocGrp}$ denote the category of local group germs with morphisms of local group germs as arrows.  Then the canonical nonstandard hull construction is a functor from $\operatorname{MGrp}$ to $\operatorname{LocGrp}$.
\end{prop}

What about the more general situation?

\begin{lemma}
Suppose that $f:G\to H$ is a continuous, \emph{open} group morphism.  Then $f(\i(U^*))\subseteq \i(V^*)$.
\end{lemma}

\begin{proof}
It is enough to prove that, for every $a\in \i(U^*)$, $f(\mu(a))=\mu(f(a))$.  By uniform continuity, $f(\mu(a))\subseteq\mu(f(a))$.  For the other direction, suppose that $f(a)\approx_l b$ but $b\notin f(\mu(a))$.  Then by saturation, there is $p\in P$ and $\epsilon>0$ such that $p(x,a)<\epsilon$ implies $b\not=f(x)$.  Since $f$ is open, there is $q\in Q$ and $\delta>0$ such that $\{y\in H \ | \ q(y,f(a))<\delta\}\subseteq f(\{x\in G \ | \ p(x,a)<\epsilon\})$.  Since $b\approx_l f(a)$, we have $q(b,f(a))<\delta$, so $b=f(x)$ for some $x\in G$ with $p(x,a)<\epsilon$, a contradiction.  (Note that we never used that $a\in \i(U^*)$ here.)  
\end{proof}

\begin{cor}
Let $\operatorname{TopGrpOp}$ denote the wide subcategory of the category $\operatorname{TopGrp}$, where the arrows are the open topological group morphisms.  Then the canonical nonstandard hull is a functor $\operatorname{TopGrpOp}\to \operatorname{LocGrp}$.
\end{cor}

A curious by-product of the above proof is the following generalization of Enflo's Proposition 1.16.

\begin{cor}  If $f:G\to H$ is an open group morphism and $G$ is locally uniform, then $H$ is locally uniform.  More precisely, if $G$ is $U$-locally uniform, then $H$ will be $f(U)$-locally uniform.  
\end{cor}

\begin{proof}
Suppose that $z,z',w,w'\in f(U)^*$ are such that $z\approx z'$ and $w\approx w'$.  Write $z=f(x)$ and $w=f(y)$, with $x,y\in U^*$.  By the proof of the above lemma, we see that $z'=f(x')$, $w=f(y')$, where $x'\in \mu(x)$ and $y'\in \mu(y)$.  Notice that $x'=x(x^{-1}x')\in (U^*)^2$ since $x^{-1}x'\in \mu(e)\subseteq U^*$.  Likewise, $y'\in (U^*)^2$.  Since multiplication on $U^2$ is uniformly continuous, we have $xy\approx x'y'$, whence, by uniform continuity of $f$, we have $$zw=f(x)f(y)=f(xy)\approx f(x'y')=f(x')f(y')=z'w'.$$
\end{proof}

Enflo's Proposition 1.16 is the special case of the above result when $H$ was taken to be a quotient $G/N$ for $N$ a closed, normal subgroup of $G$.

\section{Uniformly NSS}

$G$ is said to be \emph{uniformly NSS (UNSS)} if there is a neighborhood $U$ of the identity such that, for every neighborhood $V$ of the identity, there is $n_V\in \n$ such that, for all $x\in G$, $x\notin V\Rightarrow x^n\notin U$ for some $n\leq n_V$; one then says that the neighborhood $U$ is uniformly free from small subgroups.  It is clear that uniformly NSS groups are NSS.  For locally compact groups, the concepts coincide:

\begin{lemma}
If $G$ is locally compact, then $G$ is uniformly NSS if and only if $G$ is NSS.
\end{lemma}

\begin{proof}
Let $U$ be a compact neighborhood of the identity containing no nontrivial subgroups.  Let $V$ be an open neighborhood of the identity and suppose, towards a contradiction, that for every $m\in \n$, there is $x\in G\setminus V$ with $x^n\in U$ for $n=1,\ldots,m$.  Then, by saturation, there is $x\in G^*\setminus V^*$ such that $x^n\in U^*$ for all $n\in \n$.  Let $y:=\st(x)\in U$; then the subgroup generated by $y$ is a subgroup of $G$ contained in $U$ which is nontrivial since $y\notin V$ (else $x\in V^*$).
\end{proof}

Suppose that $G$ is uniformly NSS as witnessed by $U$.  Set $$\frac{1}{n}U:=\{x\in G \ : \ x^k\in U \text{ for }k=1,\ldots,n\}.$$  Then clearly $(\frac{1}{n}U \ : \ n\geq 1)$ is a neighborhood base for the identity.  Consequently, \emph{$G$ is metrizable}.

There exist groups which are NSS but not uniformly NSS.  For example, let $G=\r^\n$ as an abstract group.  Equip $G$ with the topology whose base is given by products of open intervals.  Then $G$ is easily seen to be NSS.  However, $G$ is not metrizable, whence it follows from the previous paragraph that $G$ is not UNSS.  (This example is Example 2.1. from \cite{enflo}).

\begin{prop}[\cite{enflo}, Theorem 2.1.1]
If $G$ is UNSS, then $G$ is locally uniform.
\end{prop}

\begin{proof}
Let $d$ be a left-invariant metric for $G$ and let $U$ be uniformly free from small subgroups.  Let $r\in \r^{>0}$ be such that $B(e,r)\subseteq U$.  We claim that multiplication is $d$-uniformly continuous when restricted to $B(e,\frac{r}{2})\times B(e,\frac{r}{2})$.  Suppose not.  Then there are $x,y,x_1,y_1\in B(e,\frac{r}{2})^*$ such that $d(x,x_1),d(y,y_1)\approx 0$ but $d(xy,x_1y_1)\not\approx 0$. Let $a=xy$, $b=x_1y$, and $c=y^{-1}$.  Then $d(a,b)\not\approx 0$ while $d(ac,bc)\approx 0$.  Set $f=a^{-1}b$.  Then $d(f,e)\not\approx 0$ but $d(fc,c)\approx 0$.  By left-invariance, for each $n\in \n$, we have $d(f^nc,f^{n-1}c)=d(fc,c)\approx 0$.  Consequently, for each $n\in \n$, we have $d(f^nc,c)\approx 0$.  Thus, for each $n\in \n$, we have $d(f^n,e)\leq d(f^nc,c)+d(c,e)<r$.  However, since $d(f,e)\not\approx 0$, there is some $n\in \n$ such that $f^n\notin U^*$; in particular, $d(f^n,e)\geq r$ for this $n$, a contradiction.
\end{proof}

Let $H$ be a local group.  We say that $H$ is \emph{uniformly free from small subgroups (UNSS)} if there is a neighborhood $U$ of the identity in $H$ so that, for any neighborhood $V$ of the identity, there is $n_V\in \n$ such that, for all $x\in G$, if $x\notin V$ and $x^{n_V}$ is defined, then $x^i\notin U$ for some $i\in \{1,\ldots,n_V\}$.  It is clear that $H$ being UNSS implies that $H$ is NSS.  It is also clear that if $H'$ is locally isomorphic to $H$, then $H$ is UNSS if and only if $H'$ is UNSS.  

\begin{lemma}\label{localUNSS}
Let $H$ be a neat, $\infty$-associative local group with Mal'cev hull $H^M$.  Then $H$ is UNSS if and only if $H^M$ is UNSS.
\end{lemma}

\begin{proof}
The ``if'' direction is obvious, so we prove the ``only if'' direction.  Suppose that $H$ is UNSS and $U$ is a symmetric open neighborhood of the identity of $H$ uniformly free from small subgroups such that $U\times U\subseteq \o_H$.  By Lemma 2.2 of \cite{vdd-Gold}, $H|U=H^M|U$.  Let $V$ be an open neighborhood of the identity in $H$ and let $x\in H^M$ be such that $x\notin V$.  We claim that $x^i\notin U$ for some $i\leq n_V$.  Suppose, towards a contradiction, that $x^i\in U$ for all $i\leq n_V$.  We then claim that $x^i$ is defined (in $H$) for all $i\leq n_V$, contradicting the definition of $n_V$.  We prove our claim by induction on $i$, the base case being trivial.  Suppose the claim is true for some $i<n_V$.  Since $x^i$ is defined, we need only show that $(x^j,x^k)\in \o_H$ for all $j,k\in \{1,\ldots,i\}$ with $j+k=i$.  However, this follows immediately from the assumption that $x^j\in U$ for each $j\leq n_V$ and $U\times U\subseteq \o_H$.
\end{proof}


Note that if $\iota:H\to H'$ is an injective morphism of local groups and $H'$ is NSS, then $H$ is NSS.  (This need not be true for UNSS.)  Thus, if $G$ is locally uniform, then some nonstandard hull of $G$ is NSS if and only if the canonical nonstandard hull of $G$ is NSS.  If $G$ is metrizable and locally uniform, then every nonstandard hull of $G$ is NSS if and only if the metric nonstandard hull of $G$ is NSS.

\begin{prop}\label{nschar}
Let $G$ be a locally uniform group.  Then the following are equivalent:
\begin{enumerate}
\item $G$ is uniformly NSS;
\item $G$ is metrizable and the metric nonstandard hull is uniformly NSS;
\item $G$ is metrizable and the metric nonstandard hull is NSS;
\item $G$ is metrizable and every nonstandard hull of $G$ is NSS.
\end{enumerate}
\end{prop}

\begin{proof}
(1)$\Rightarrow$ (2):  Suppose that $G$ is uniformly NSS.  Fix a left-invariant metric $d$ on $G$.  Without loss of generality, we may suppose that $r\in \r^{>0}$ has been chosen so that $U:=B(e,r)$ is uniformly free from small subgroups and $G$ is $U$-locally uniform.  We will show that $\hat{U}_d$ is uniformly free from small subgroups.  Fix $\delta>0$ and suppose that $\hat{d}([x],[e])\geq \delta$.  Let $n_\delta\in \n$ be such that, for all $y\in G$, if $d(y,e)\geq \frac{\delta}{2}$, then $y^i\notin U$ for some $i\leq n_\delta$.  By transfer, there is $i\leq n_\delta$ such that $x^i\notin U^*$, whence $[x]^{n_\delta}$ isn't defined by Remark \ref{defined}.   


(2) $\Rightarrow$ (3):  Trivial.

(3) $\Leftrightarrow$ (4):  This follows from the remarks preceding the theorem.

(3) $\Rightarrow (1)$:  Suppose that $G$ is metrizable and $\hat{U}_d$ is NSS.  Without loss of generality, suppose that $U=B(e,r)$ for some $r\in \r^{>0}$.  Take $\epsilon\in(0,r)$ such that $U':=\{[x]\in \hat{U}_d \ | \ \hat{d}([x],[e])<\epsilon\}$ contains no nontrivial subgroups.  We claim that $W=B(e,\frac{\epsilon}{2})\subseteq G$ is uniformly free from subgroups.  Suppose, towards a contradiction, that there is an open neighborhood $V$ of the identity in $G$ such that, for all $n\in \n$, there is $x_n\in G\setminus V$ such that $x_n^m\in W$ for all $m\leq n$.  Then, by saturation, there is $x\in G^*$ such that $x\notin V^*$ and $x^m\in W^*$ for all $m\in \n$.  Since $W^*\subseteq U_{d,f}$, we have that $[x]^m$ is defined and in $U'$ for all $m\in \n$.  Since $x\notin V^*$, we have that $[x]\not=[e]$.  This contradicts the fact that $U'$ contains no nontrivial subgroups.      
\end{proof}

\begin{rmks}

\

\begin{enumerate}
\item The fact that a metrizable, locally uniform group is uniformly NSS if and only if its metric nonstandard hull is NSS is an example of a familiar phenomenon in nonstandard analysis, namely the uniform version of a notion for an object $X$ is equivalent to the ordinary notion of the concept for some nonstandard object associated to $X$, e.g. $X^*$ or some nonstandard hull of $X$.  A recent interesting example of this phenomenon was observed by David Ross, who showed that a group $G$ is uniformly amenable if and only if $G^*$ is amenable.
\item The proof of (1)$\Rightarrow$(2) in the above lemma also shows that if $G$ is UNSS, then the canonical nonstandard hull of $G$ is UNSS.  Also, if the canonical nonstandard hull of $G$ is UNSS, then $G$ is metrizable.  Indeed, consider the map $\iota:U\to \hat{U}$.  Then given $p\in P$ and $r\in \r^{>0}$, we have that $\iota^{-1}(V_{\sta p,r})=V_{p,r}$.  Since $\iota$ is continuous and $\hat{U}$ is metrizable (as its Malcev hull is UNSS), it follows that $U$ has a countable base at the identity, whence $G$ has a countable base at the identity, and is thus metrizable.  
\item By Lemma \ref{localUNSS}, we can replace the local nonstandard hulls by their global counterparts in the above proposition.
\end{enumerate}
\end{rmks}

\begin{question}
If the canonical nonstandard hull of $G$ is (uniformly) NSS, then is any (or even some) metric nonstandard hull of $G$ UNSS?
\end{question}

\begin{question}
Is it possible to find a locally uniform group $G$ so that $\hat{U}$ is NSS but not uniformly NSS?  We observe that such a $G$ would have to be NSS but not uniformly NSS (and not metrizable).
\end{question}

We use Proposition \ref{nschar} to give a simple proof of \cite{enflo}, Theorem 2.2.2.

\begin{thm}
If $G$ is a uniformly NSS group, then there is a neighborhood $U$ of $1$ in $G$ so that, for all $x,y\in U$, if $x^2=y^2$, then $x=y$.
\end{thm}

\begin{proof}
It is enough to show that for all $x,y\in \mu$, if $x^2=y^2$, then $x=y$.  Suppose this is not the case.  Let $x,y\in \mu$ be such that $x^2=y^2$ but $x\not=y$.  Let $a:=xy^{-1}\in \mu \setminus \{e\}$ and note that $y^{-1}a^ky=a^{-k}$ for all $k$.  Construct the local group $\hat{U}_d$ as above.  Since $\hat{U}_d$ is NSS, we can choose $\eta \in (0,\epsilon)$ so that $\{[z] \in \hat{U}_d \ | \ \dh([z],[e])\leq \eta\}$ contains no nontrivial subgroups.  Since $G$ is NSS, we know that $a^k\notin \mu$ for some $k$.  

\

\noindent \textbf{Claim:}  There is a $k$ such that $d(a^k,e)>\eta$.  Suppose this is not the case.  Choose $k$ such that $a^k\notin \mu$.  Let $b:=[a^k]\in \hat{U}_d$.  Then $b$ generates a nontrivial subgroup of $\hat{U}_d$ contained in $\{[z] \in \hat{U}_d \ | \ \dh([z],[e])\leq \eta\}$, a contradiction.

\

\noindent By the claim, we can choose $k$ maximal such that $d(a^i,e)\leq \eta$ for all $i\leq k$.  Let $c:=[a^k]\in \hat{U}_d$.  Then $c\not= [e]$, but $c=c^{-1}$, whence the nontrivial subgroup $\{1,c\}$ of $\hat{U}_d$ is contained in $\{[z] \in \hat{U}_d \ | \ \dh([z],[e])\leq \eta\}$, a contradiction.  
\end{proof}

\

\noindent \textbf{Infinite-dimensional Lie groups and UNSS}

\

One might wonder which infinite-dimensional Lie groups are uniformly NSS.  As far as the locally exponential ones are concerned, not many.

\begin{lemma}
Suppose that $G$ is a topological group, $E$ is a locally convex space, and $\exp:E\to G$ is a continuous map that is a local homeomorphism at $0$.  Further suppose that $\exp(ka)=\exp(a)^k$ for all $a\in E$ and $k\in \z$.  Then $G$ is uniformly NSS if and only if $E$ is normable.  
\end{lemma}

\begin{proof}
The ``if'' direction is well-known, but we give the proof here for the sake of completeness.  Fix open $U'\subseteq E$ and $U\subseteq G$, neighborhoods of $0$ and $1$ respectively, so that $\exp\upharpoonright U'$ is a homeomorphism from $U'$ onto $U$.  By rescaling the norm if necessary, we may assume that $U'=\{x\in E \ | \ \|x\|<1\}$.  Fix an open neighborhood $V$ of $1$.  Without loss of generality, we may assume that $V\subseteq U$.  Take $\lambda>0$ be such that $\{x\in E \ | \ \|x\|<\lambda\}\subseteq \exp^{-1}(V)$.  Take $n_V\in \n$ such that $n_V\cdot \lambda\geq 1$.  Suppose that $y\in U\setminus V$; write $y=\exp(x)$, where $x\in U'$ and $\|x\|\geq \lambda$.  Then $y^{n_V}=\exp(n_V\cdot x)$; since $\| n_V\cdot x\|\geq 1$, we have that $y^{n_V}\notin U$.              

Conversely, suppose that $G$ is uniformly NSS.  Let $U$ be an open, balanced, convex neighborhood of $0$ in $E$ so that $V:=\exp(U)$ is an open neighborhood of the identity of $G$ uniformly free from small subgroups and $\exp|U:U\to V$ is a homeomorphism.  It suffices to prove that $U$ is a bounded set.  Let $W$ be an open neighborhood of $0$ in $E$.  Without loss of generality, $W\subseteq U$.  Let $W':=\exp(W)$ and set $n:=n_{W'}$.  We claim that $\frac{1}{n}U\subseteq W$.  Suppose that $y\notin W$.  Then $\exp(y)\notin W'$, so $\exp(y)^i\notin V$ for some $i\in \{1,\ldots,n\}$.  For this $i$, $\exp(iy)\notin V$, so $iy\notin U$, so $y\notin \frac{1}{i}U$.  Since $U$ is balanced, $\frac{1}{n}U\subseteq \frac{1}{i}U$, whence $y\notin \frac{1}{n}U$.
\end{proof}

The previous lemma is in a similar spirit to a result of Glockner (which appears in the introduction of \cite{glockner}) stating that, under the same hypotheses on $G$ and $E$, we have that $G$ is NSS if and only if $E$ admits a continuous norm.  In the same paper, Gl\"ockner proves that direct limits of finite-dimensional Lie groups are NSS.  It is not clear to us if his methods answer the following

\begin{question}
If $G$ is a direct limit of finite-dimensional Lie groups, is $G$ UNSS?
\end{question} 

\noindent We now describe a large class of Lie groups which need not be locally exponential and which are uniformly NSS, namely the \textbf{strong ILB-Lie groups}.

\begin{df}(Omori, \cite{omori})
A \textbf{Sobolev chain} is a sequence $(E_n|n\geq d)$ of Banach spaces, where each $E_n$ is a Banach space and $E_{n+1}$ is continuously, linearly, and densely embedded in $E_n$ for all $n\geq d$.  We let $E:=\varprojlim E_n$, a Frechet space.  Without loss of generality, we may assume each $E_{n+1}$ is a (dense) subspace of $E_n$ and that $E=\bigcap E_n$, equipped with the inverse limit topology. 
\end{df}

\begin{df}(Omori, \cite{omori})
A topological group $G$ is called a \textbf{strong ILB-Lie group modeled on the Sobolev chain $(E_n:n\geq d)$} if the following conditions are satisfied:
\begin{enumerate}
\item[(N1)] There exists an open neighborhood $U$ of $0$ in $E_d$ and a homeomorphism $\psi$ of $U\cap E$ (equipped with the relative topology from $E$) onto an open neighborhood $\widetilde{U}$ of $e$ in $G$ such that $\psi(0)=e$;
\item[(N2)] There exists an open neighborhood $V$ of $0$ in $E_d$ such that $\psi(V\cap E)$ is symmetric and $\psi(V\cap E)^2\subseteq \psi(U\cap E)$; 
\item[(N3)] Let $\eta:(V\cap E) \times (V\cap E)\to U\cap E$ be defined by $$\eta(u,v)=\psi^{-1}(\psi(u)\psi(v)).$$  Then for every $n\geq d$ and $l\geq 0$, $\eta$ can be extended to a $C^l$-mapping $\eta:(V\cap E_{n+l})\times (V\cap E_n)\to U\cap E_n$;
\item[(N4)] For $v\in V\cap E$, let $\eta_v:V\cap E \to U\cap E$ be defined by $\eta_v(u)=\eta(u,v)$.  Then for every $v\in V\cap E$ and every $n\geq d$, $\eta_v$ can be extended to a $C^\infty$-mapping $\eta_v:V\cap E_n\to U\cap E_n$;
\item[(N5)] Let $\theta_n:E_n\times (V\cap E) \times (V\cap E) \to E_n$ be defined by $$\theta_n(w,u,v)=(d\eta_v)(u)(w).$$  Then for every $l\geq 0$, $\theta_n$ can be extended to a $C^l$-mapping $\theta_n:E_{n+l}\times (V\cap E_{n+l})\times (V\cap E_n)\to E_n$;
\item[(N6)] Let $\iota:V\cap E \to V\cap E$ be defined by $\iota(u)=\psi^{-1}(\psi(u)^{-1})$.  Then for every $n\geq d$ and $l\geq 0$, $\iota$ can be extended to a $C^l$-mapping $\iota:V\cap E_{n+l}\to V\cap E_n$;
\item[(N7)] For any $g\in G$, there exists an open neighborhood $W_g$ of $0$ in $E_d$ such that $g^{-1}\psi(W_g\cap E)g\subseteq \psi(U\cap E)$.  Let $A_g:W_g\cap E\to U\cap E$ be defined by $A_g(u)=\psi^{-1}(g^{-1}\psi(u)g)$.  Then for every $n\geq d$, $A_g$ can be extended to a $C^\infty$-mapping $A_g:W_g\cap E_k\to U\cap E_k$.
\end{enumerate}
\end{df}

\noindent If all $E_n$ are Hilbert spaces, then we speak of \textbf{strong ILH-Lie groups}.  Note that by (N1), a strong ILB-Lie group is a locally convex Lie group modeled on a Fr\'echet space.

There are many natural examples of strong ILB-Lie groups:
\begin{fact}[see Section 3.3 of \cite{neeb}]
Let $M$ be a smooth compact manifold.  Then the following are strong ILH-Lie groups:
\begin{enumerate}
\item $\Diff(M)$.
\item $\Diff(M,\omega):=\{\varphi \in \Diff(M) \ | \ \varphi^*\omega=\omega\}$, where $\omega$ is a symplectic $2$-form on $M$, a volume form on $M$, or a contact form on $M$.
\item $\Diff(M,N):=\{\varphi \in \Diff(M) \ | \ \varphi(N)=N\}$, where $N$ is a closed submanifold of $M$.
\item $\Diff_K(M):=\{\varphi \in \Diff(M) \ | \ \varphi \circ k =k \circ \varphi \text{ for all }k\in K\}$, where $K$ is a compact subgroup of $\Diff(M)$.
\end{enumerate}
\end{fact}  

\noindent Omori \cite{omori} shows that strong-ILB Lie groups are NSS.  We repeat his proof here as it actually shows that strong-ILB Lie groups are uniformly NSS.

\begin{prop}
If $G$ is a strong ILB-Lie group, then $G$ is uniformly NSS.
\end{prop}

\begin{proof}
By (N3), we have the $C^1$-map $\eta:(V\cap E_{d+1})\times V\to U$.  For $u\in V\cap E_{d+1}$, let $\rho_u:V\to U$ be defined by $\rho_u(v)=\eta(u,v)$.  Fix $\epsilon\in (0,1)$.  Then there are open neighborhoods $W_{d+1}$ and $W_d$ of $V\cap E_{d+1}$ and $V$ respectively such that for all $u\in W_{d+1}$, $v\in W_d$, and $w\in E_d$, we have $$\|(d\rho_u)(v)(w)-w\|_d\leq \epsilon \|w\|_d,$$ where $\|\cdot \|_d$ denotes the norm of $E_d$.  Without loss of generality, we may assume that $W_{d+1}$ is contained in the open ball in $E_d$ of radius $M$.  The following claim finishes the proof of the theorem.

\

\noindent \textbf{Claim:}  $\psi(W_{d+1}\cap E)$ is uniformly free from subgroups.

\noindent \textbf{Proof of claim:}  For $v\in V$, define $\eta_v:V\cap E_{d+1}\to U$ by $\eta_v(u)=\eta(u,v)$.  Let $u\in V\cap E_{d+1}$ and suppose that $\eta_u^i(u)\in W_{d+1}$ for all $i\in \{1,\ldots,m\}$.  Since we have 
$$\eta_u^m(u)-\eta_u^{m-1}(u)=\int_0^1 (d\rho)_{\eta_u^{m-1}(u)}(tu)(u)dt,$$
it follows that $$\eta_u^m(u)=mu+\sum_{i=0}^{m-1} \int_0^1 ((d\rho)_{\eta_u^{i}(u)}-I)(tu)(u)dt.$$
By the above estimates, we see that $\|\eta_u^m(u)\|_d\geq m(1-\epsilon) \|u\|_d$. 

In order to finish the claim, let us suppose $Z$ is a neighborhood of $e$ in $G$ and $\psi(u)\in \psi(W_{d+1}\cap E)\setminus Z$.  Fix $M'\in \r$ such that if $\psi(u)\notin Z$, then $\|u\|_d\geq M'$.  Choose $n_Z$ so that $n_Z(1-\epsilon)M'\geq M$.  It then follows that $\eta_u^i(u)\notin W_{d+1}\cap E$ for some $i\leq n_Z$.    
\end{proof}

As a result, strong ILB-Lie groups are locally uniform.  We should mention that there are examples of strong ILB-Lie groups that are not locally exponential; for example, by pg. 343 of \cite{neeb}, the group $\operatorname{Diff}(M)$, where $M$ is a compact manifold, is not locally exponential.

\begin{question}
Other than the locally compact Lie groups, additive groups of locally convex spaces, and UNSS groups, are there any other examples of infinite-dimensional Lie groups that are locally uniform?
\end{question}

Karl Hermann-Neeb suggested that the unit groups of continuous inverse algebras might be locally uniform, but we were not able to establish this fact.






\section{An example:  the group of units of a unital Banach algebra}

Suppose that $\A$ is a unital Banach algebra and $G=U(\A)$ is the group of units of $\A$.  Recall that $G$ is an open neighborhood of $1$ in $\A$ and that $\{y\in \A\ | \ \|y-1\|<1\}\subseteq G$.  Since $G$ is a Banach-Lie group, $G$ is UNSS and locally uniform. We can thus construct the metric nonstandard hull of $G$.  However, there is also the Banach algebra nonstandard hull $\hat{\A}$ of $\A$.  The goal of this section is to understand the relationship between these two nonstandard hulls.  

We fix $M\in \r^{>0}$ such that $\|xy\|\leq M\|x\|\|y\|$ for all $x,y\in \A$; this is possible by the uniform boundedness principle.

As usual, for $x,y\in G^*$, we write $x\approx_ly$ to mean that $x^{-1}y\in \mu$.  For $x,y\in \A^*$, write $x\approx_\A y$ to mean $\|x-y\|\approx 0$.  Let $d$ be a left-invariant metric on $G$ compatible with the topology; in particular, for $a\in G^*$, we have $$a\approx_l 1\Leftrightarrow d(a,1)\approx 0\Leftrightarrow \|a-1\|\approx 0 \Leftrightarrow a\approx_\A 1.$$  Let $\epsilon\in\r^{>0}$ be small enough so that, setting $W:= \{x\in G\ | \ d(x,1)<\epsilon\}$, we have $W\subseteq \{x\in \A \ | \ \|x-1\|<1\}\subseteq G$ and $G$ is $W$-locally uniform.  Then $W^*\subseteq \A_f:=\{x\in \A^* \ | \ \|x\|\in \r_f\}$.  Moreover, for $x,y\in W^*$, we have $$x\approx_ly \Leftrightarrow x^{-1}y\in \mu\Leftrightarrow \|x^{-1}y-1\|\approx 0\Leftrightarrow \|x-y\|\approx 0.$$  Also notice that, if $x\in \i(W^*)$ and $y\in \A^*$ is such that $x\approx_\A y$, then $y\in W^*$.  Indeed, we first observe that $y\in G^*$.  To see this, observe that $d(x^{-1},1)=d(1,x)<\epsilon$, so $x^{-1}\in \A_f$.  We now have that $$\|x^{-1}y-1\|=\|x^{-1}(y-x)\|\leq M\|x^{-1}\|\|y-x\|\approx 0,$$ whence $x^{-1}y\in G^*$.  It follows that $y=x\cdot (x^{-1}y)\in G^*$.  Now take $\delta\in \r^{>0}$ so that $d(x,1)<\epsilon-\delta$.  Now $\|x-y\|\approx 0\Rightarrow \|x^{-1}y-1\|\approx 0$, whence $d(x,y)=d(x^{-1}y,1)<\delta$ and hence $d(y,1)<\epsilon$.  Consequently, for $x\in \i(W^*)$, we have that $$[x]:=\{y\in W^* \ | \ x\approx_ly\}=\{y\in \A^* \ | \ \|x-y\|\approx 0\}.$$  In particular, we have that $\hat{W}_d\subseteq \hat{A}$.  Moreover, since $\hat{W}_d$ is globally inversional, we have that $\hat{W}_d\subseteq U(\hat{\A})$.  

We claim that $\hat{W}_d$ is an open subset of $U(\hat{A})$.  Fix $[x]\in \hat{W}_d$.  Fix $\delta>0$ so that $d(x,1)<\epsilon-\delta$.  Fix $\alpha>0$ small enough so that, for all $z\in \A$, if $\|1-z\|<\alpha$, then $z\in G$ and $d(1,z)<\frac{\delta}{2}$.  Fix $\eta>0$ small enough so that $\eta\cdot \|x^{-1}\|<\alpha$.  Now suppose that $[y]\in \hat{A}$ is such that $\|[x]-[y]\|<\eta$.  Then $\|x-y\|<\eta$, so $\|1-x^{-1}y\|=\|x^{-1}(x-y)\|<\alpha$.  Thus, $x^{-1}y\in G^*$ and $d(1,x^{-1}y)<\frac{\delta}{2}$, so $d(x,y)<\frac{\delta}{2}$, whence $d(y,1)<\epsilon-\frac{\delta}{2}$ and thus $y\in \i(W^*)$ and $[y]_d\in \hat{W}$.  We have thus proven:

\begin{prop}
$\hat{W}_d$ is a restriction of $U(\hat{A})$ to a symmetric neighborhood of the identity.
\end{prop}

\section{Relationship with Pestov's Nonstandard Hull Construction}

As Banach-Lie groups are uniformly NSS (and hence locally uniform), we can consider their nonstandard hulls.  Pestov \cite{pestov} also has a nonstandard hull construction for Banach-Lie groups; his nonstandard hull is once again a Banach-Lie group.  In this section, we show that the metric nonstandard hull of a Banach-Lie group is locally isomorphic to the nonstandard hull that Pestov constructs.  The reason that the aforementioned fact is interesting is that our nonstandard hull construction is purely topological, while Pestov's construction involves some nontrivial Lie theory.  Moreover, we show that, for a suitable choice of a locally uniform neighborhood, the corresponding global nonstandard hull is the universal covering group of Pestov's nonstandard hull. 

Throughout this section, $G$ denotes a Banach-Lie group with Banach-Lie algebra $\g$.  A norm $\| \cdot \|$ on $\g$ is fixed so that $\|[x,y]\|\leq \|x\|\|y\|$ for all $x,y\in \g$.  Recall that we have the exponential map $\exp:\g\to G$, which is a local diffeomorphism.  We let $\log:G\rightharpoonup \g$ be the inverse of $\exp$.  We will need the following two lemmas.

\begin{lemma}\label{expuc}
$\exp:\g\to G$ is locally uniformly continuous.
\end{lemma}

\begin{proof}
We first show that $x*y:=\exp^{-1}(\exp(x)\cdot \exp(y)):\g\times \g\rightharpoonup \g$ satisfies $x*y=x+y+O(\|x\|\|y\|)$ for $x,y$ small enough.  Indeed, setting $m(x,y):=x*y$, we have $m(x,y)=x+y+\int_0^1 (1-t)dm(tx,ty)(x,y)dt$.  Now $\|dm(tx,ty)(x,y)\|\leq \|dm(tx,ty)\|\cdot \|(x,y)\|\leq \|dm(tx,ty)\|\cdot \|x\|\cdot \|y\|$.  Suppose that $V\times V$ is the domain of $m$.  Since the map $dm:V\times V\times \g\times \g\to \g$ is continuous, if $V$ is chosen small enough, we have that, for any $(a,b)\in \g\times \g$, $\sup \{\|dm(tx,ty)(a,b)\| \ : \ t\in [0,1], \ x,y\in V\}<\infty$.  Thus, by the uniform boundedness principle, there is $M\in \r^{>0}$ such that $\|dm(tx,ty)\|\leq M$ for all $t\in [0,1]$ and $x,y\in V$.  It follows that $$\left\|\int_0^1 (1-t)dm(tx,ty)(x,y)dt\right\|\leq M\cdot \|x\|\cdot \|y\|.$$ 

As a result, we have that $x*(-y)=O(\|x-y\|)$, say $\|x*(-y)\|\leq C\|x-y\|$.  Let $U$ be a neighborhood of the identity and choose $\epsilon>0$ small enough so that $\|z\|<\epsilon$ implies $\exp(z)\in U$.  Then if $\|x-y\|<\frac{\epsilon}{C}$, we have $$\exp(x)\exp(y)^{-1}=\exp(x*(-y))\in U.$$
\end{proof}

\begin{lemma}
$\log:G\rightharpoonup \g$ is locally uniformly continuous.
\end{lemma}

\begin{proof}
Given $\epsilon>0$, we need a neighborhood $U$ of $1$ in $G$ such that whenever $a,b$ are small enough and $ab^{-1}\in U$, then $\|\log a -\log b\|< \epsilon$.  We showed in the previous lemma that $x*y:=x+y+O(\|x\|\cdot \|y\|)$ for $x,y\in \g$ small enough.  We then get $$x=(x*(-y))*y=(x*(-y))+y+z,$$ where $\|z\|\leq C'\cdot \|x*(-y))\|\cdot \|y\|$.  Since $\|y\|$ is bounded, it then follows that there is a constant $C$ such that $\|x-y\|\leq C\| x*(-y)\|.$   

Choose a neighborhood $U$ of $1$ such that $c\in U$ implies $\|\log c\|< \frac{\epsilon}{C}$.  Suppose $a,b\in G$ are sufficiently close to the identity and $ab^{-1}\in U$.  Then if $a=\exp x$ and $b=\exp y$, we have $\log(ab^{-1})=x*(-y)$, so $$\|\log a-\log b\|=\|x-y\|\leq C\|x*(-y)\|=C\|\log(ab^{-1})\|< \epsilon.$$
\end{proof}

We now summarize Pestov's construction of the nonstandard hull of $G$, which we will denote by $\gh_{\VP}$.  Define $\mu_\g:=\{x\in \g^* \ | \ \|x\|\in \mu(0)\}$, which is a Lie ideal of the Lie algebra  $\g_f:=\{x\in \g^* \ | \ \|x\|\in \r^f\}$.  Let $\hg:=\g_f/\mu_\g$ be the quotient Lie algebra.  Let $\pi_\g:\g_f\to \hg$ be the quotient map and define a norm on $\hg$ by $\|\pi_\g(x)\|:=\st(\|x\|)$.  One then defines $G_{f,\VP}:=\bigcup_n(\exp V)^n\subseteq G^*$ where $V$ is any ball of finite, noninfinitesimal radius in $\g^*$.  (This turns out to be independent of $V$.)  Pestov shows (using some nontrivial Lie theory) that $\mu$ is a normal subgroup of $G_{f,\VP}$.  We set $\gh_{\VP}:=G_{f,\VP}/\mu$ and let $\pi_G:G_{f,\VP}\to \gh_{\VP}$ be the quotient map.  We define $\hat{\exp}:\hg \to \gh_{\VP}$ by $\hat{\exp}(\pi_\g(x)):=\pi_G(\exp(x))$.  Pestov shows that there is a neighborhood of $0$ in $\hg$ such that $\hat{\exp}$ restricted to this neighborhood is injective and that there is a unique structure of a Banach-Lie group on $\gh_{\VP}$ such that $\hat{\exp}$ becomes a local diffeomorphism.   

Fix $\delta'>0$ and $M\in \r^{>0}$ such that, if $\max(\|x\|,\|y\|)<\delta'$, then $\|(x*y)-(x+y)\|\leq M\cdot \|x\|\cdot \|y\|$ (see the proof of Lemma \ref{expuc}).  We now fix $\delta>0$ satisfying $\delta<\min(\delta',\frac{1}{M})$ and so that, setting $$V:=\{g\in \g \ | \ \|g\|<\delta\} \text{, }Z:=\{g\in \hg \ | \ \|g\|<\delta\}, \text{ and }W:=\exp(V),$$ we have:
\begin{itemize}
\item $\exp|V:2V\to \exp(V)$ is a diffeomorphism onto a neighborhood of $e$ in $G$;
\item $\hat{\exp}|3Z:3Z\to \hat{\exp}(3Z)$ is a diffeomorphism onto an open neighborhood of $\pi_G(e)$ in $\gh_{\VP}$;
\item  $\exp|V$ is an isomorphism of uniform spaces (which is possible by the previous two lemmas);
\item $G$ is $W$-locally uniform.
\end{itemize}

Fix a left-invariant metric $d$ on $G$ such that $W\subseteq B_d(e,1)$. 
By transfer, $W^*= \exp(V^*)\subseteq G_{f,\VP}$.  We can thus consider the injective map $$i:\hat{W}_d \to \gh_{\VP}, \quad i([x])=\pi_G(x).$$  We claim that $i(\hat{W}_d)\subseteq \hat{\exp}(Z)$.  To see this, take $[x]\in \hat{W}_d$, so $x=\exp(u)$ for some $u\in V^*$.  Then $i([x])=\pi_G(x)=\hat{\exp}(u+\mu_\g)$.  It remains to show that $\st \|u\|<\delta$, that is, that $u\in \i(V^*)$.  Since $\exp|V$ is an isomorphism of uniform spaces, we have $\mu(u)=\log \mu(x)$.  Since $\mu(x)\subseteq W^*$, we have $\mu(u)\subseteq V^*$.

In what follows, we view $\hat{\exp}(Z)$ as the local group $\hat{G}_{\VP}|\hat{\exp}(Z)$.
\begin{thm}
$i:\hat{W}_d\to \hat{\exp}(Z)$ is an isomorphism of local groups.  Moreover, $\hat{G}_W$ is the universal covering group of $\hat{G}_{\VP}$.
\end{thm}

\begin{proof}
We first show that $i$ is onto.  Consider $\hat{\exp}(a+\mu_\g)$, where $a+\mu_\g\in Z$, so $\st(\|a\|)<\delta$.  Then $\hat{\exp}(a+\mu_\g)=\pi_G(\exp(a))$.  We need to prove that $\exp(a)\in \i(W^*)$.  Since $\exp|V$ is an isomorphism of uniform spaces, $\mu(\exp(a))=\exp(\mu(a))$; since $a\in \i(V^*)$, we have $\mu(\exp(a))\subseteq W^*$.

If $([x],[y])\in \o$, then $$i([x]\cdot [y])=i([xy])=\pi_G(xy)=\pi_G(x)\cdot \pi_G(y)=i([x])\cdot i([y]).$$  A similar argument shows that $i$ respects inversion.  Suppose now that $\hat{\exp}(a+\mu_\g)\cdot \hat{\exp}(b+\mu_g)=\hat{\exp}(c+\mu_g)$, where $c+\mu_\g\in Z$.  Then $\exp(a)\cdot \exp(b)\approx \exp(c)\in \i(W^*)$, whence $([\exp(a)],[\exp(b)])\in \Omega$.  It follows that $i$ is a strong morphism of discrete local groups.

It remains to prove that $i$ is a homeomorphism.  By local homogeneity (Lemma 2.16 of \cite{Gold}), it suffices to prove that $i$ is continuous and open at $[e]$.  Towards this end, let $\mathcal{O}\subseteq \hat{\exp}(Z)$ be an open neighborhood of $\pi_G(e)$; we must prove that $i^{-1}(\mathcal{O})$ is open in $\hat{W}_d$.  Fix $\gamma \in (0,\delta)$ be so that if $\|g+\mu_\g\|<\gamma$, then $\hat{\exp}(g+\mu_\g)\in \mathcal{O}$.  It suffices to find $\eta>0$ such that if $[x]\in \hat{W}$ is such that $\dh([x],[e])<\eta$, then $\|\hat{\exp}^{-1}(\pi_G(x))\|<\gamma$.  Fix $\gamma_1\in (0,\gamma)$ and let $V_1:=\{g\in \g \ | \ \|g\|<\gamma_1\}$.  Then $\exp(V_1)$ is a neighborhood of $e$ in $G$ and so we can choose $0<\eta<1$ so that $B(e,\eta)\subseteq \exp(V_1)$.  We will show that this is the desired $\eta$.  Suppose $\dh([x],[e])<\eta$.  Then $d(x,e)<\eta$, whence $x=\exp(x')$ for $x'\in V_1^*$.  Then $\pi_G(x)=\pi_G(\exp(x'))=\hat{\exp}(\pi_\g(x'))$.  Since $\|\pi_\g(x')\|=\st(\|x'\|)\leq \gamma_1<\gamma$, we have that $\|\hat{\exp}^{-1}(\pi_G(x))\|=\|\pi_\g(x')\|<\gamma$.    

We now show that $i$ is an open map.  Fix $\alpha\in (0,1]$.  It suffices to show that $i(\{[x]\in \hat{W}_d \ : \ \hat{d}([x],[e])<\alpha\})$ is an open subset of $\hat{G}_{\VP}$.  Fix $[x]\in \hat{W}_d$ such that $\hat{d}([x],[e])<\alpha$.  Since $\pi_G(x)\in \hat{\exp}(Z)$, we can write $\pi_G(x)=\hat{\exp}(a+\mu_\g)=\pi_G(\exp(a))$ for some $a+\mu_\g\in Z$.  Fix $\beta>0$ so that $\st(d(x,e))+\beta<\alpha$ and fix $\eta>0$ small enough so that, for any $c,d\in V$ with $\|c-d\|<\eta$, we have $d(\exp(c),\exp(d))<\beta$ (by uniform continuity of $\exp|V$) and further satisfying $\st(\|a\|)+\eta<\delta$.  Let $$\cO=\hat{\exp}(\{b+\mu_\g \in \hat{\g} \ | \ \|(a+\mu_\g)-(b+\mu_\g)\|<\eta\}).$$  Then $\cO$ is an open subset of $\hat{\exp}(Z)$.  We claim that if $\pi_G(y)\in \cO$, then $\pi_G(y)\in i(\{[x]\in \hat{W} \ : \ \hat{d}([x],[e])<\alpha\})$.  Take $b+\mu_\g$ so that $\st(\|a-b\|)<\eta$ and $\pi_G(y)=\hat{\exp}(b+\mu_\g)=\pi_G(\exp(b))$.  Then by transfer, we have $d(\exp(a),\exp(b))<\beta$, whence $d(y,e)\approx d(\exp(b),e)\leq d(\exp(a),e)+\beta$ and thus $\st(d(y,e))<\alpha$.

We now prove the moreover part.  Set $\hat{G}_{\VP}^{\operatorname{o}}$ to be the identity component of $\hat{G}_{\VP}$.  Then $\hat{G}_{\VP}^{\operatorname{o}}$ is a connected, locally simply connected group.  Moreover, $\hat{\exp}(3Z)$ is a simply connected open neighborhood of the identity in $\hat{G}_{\VP}^{\operatorname{o}}$ and $\hat{\exp}(Z)$ is a connected open neighborhood of the identity in $\hat{G}_{\VP}^{\operatorname{o}}$ satisfying $\hat{\exp}(Z)^2\subseteq \hat{\exp}(3Z)$.  To see this last part, consider $\hat{\exp}(\pi_\g(a)),\hat{\exp}(\pi_\g(b))\in \hat{\exp}(Z)$.  Then $\hat{\exp}(\pi_\g(a))\cdot \hat{\exp}(\pi_\g(b))=\hat{\exp}(\pi_\g(a*b))$; it remains to see that $\st \|a*b\|<3\delta$.  However, $$\st \|a*b\|\leq \st\|a\|+\st\|b\|+M\cdot \st\|a\|\cdot \st\|b\|<3\delta$$ by the choice of $\delta$.  Thus, by Proposition \ref{unimal}, the Mal'cev hull of $\hat{\exp}(Z)$ is the universal covering group of $\hat{G}_{\VP}^{\operatorname{o}}$.  The desired result follows from the fact that $\hat{\exp}(Z)$ is isomorphic to $\hat{W}_d$ and that $\hat{G}_{\VP}^{\operatorname{o}}$ is locally isomorphic to $\hat{G}_{\VP}$.
\end{proof}

\end{document}